\documentclass{article}

\usepackage{amsthm}
\usepackage{amsmath, amssymb}
\usepackage{hyperref}
\usepackage{tikz}
\usetikzlibrary{matrix}
\usepackage[matrix, arrow, curve]{xy}

\newtheorem{Thm}{Theorem}[section]

\newtheorem{Rem}[Thm]{Remark}

\title{A symplectic version of Suslin's $n!$-theorem}

\author{Tariq Syed  \\
	Institut f{\"u}r Mathematik\\
	Johannes Gutenberg-Universit{\"a}t Mainz\\
	Staudingerweg 9\\
	55128 Mainz, Germany\\
	tariq.syed@gmx.de}

\date{\today}
\begin{document}

\maketitle

\begin{abstract}
We prove symplectic versions of Suslin's famous $n!$-theorem for algebras over quadratically closed perfect fields of characteristic $\neq 2$ and for algebras over finite fields of characteristic $\neq 2$.
\end{abstract}

\tableofcontents

\section{Introduction}

The question under which circumstances stably free modules over commutative rings are actually free has stimulated a wealth of research over many decades and has led to the development of the basic calculus of unimodular rows in the second half of the last century (cf. \cite[Chapter III]{L}). Indeed, if $R$ is any commutative ring with unit and $n \geq 0$ an integer, a unimodular row of length $n+1$ is a row vector $(a_{0},..,a_{n})$ of length $n+1$ with $a_i \in R$, $0 \leq i \leq n$, such that $\langle a_{0},...,a_{n}\rangle = R$; any such row vector corresponds to an epimorphism $R^{n+1} \rightarrow R$ whose kernel is a stably free $R$-module of rank $n$. This stably free $R$-module is free if and only if the corresponding unimodular row is the first row of an invertible $(n+1)\times(n+1)$-matrix over $R$. Therefore it became of interest to find general criteria for a unimodular row to be completable to an invertible matrix. As a highlight of this development, Suslin proved the following beautiful and remarkable result on the completability of unimodular rows (cf. \cite[Theorem 2]{S1}):\\\\
\textbf{Theorem (Suslin).} Let $R$ be a commutative ring, $n \geq 0$ an integer and $(a_{0},...,a_{n})$ be any unimodular row of length $n+1$ over $R$. Furthermore, let $r_{0},...,r_{n} \geq 1$ be integers such that $n!$ divides $r_{0} \cdot ... \cdot r_{n}$. Then the unimodular row $(a_{0}^{r_{0}},...,a_{n}^{r_{n}})$ is the first row of an invertible matrix $\varphi \in GL_{n+1}(R)$.\\\\
This result is now known as Suslin's $n!$-theorem (cf. \cite[Chapter III, \S 4]{L}). The theorem was substantially used in the proofs of celebrated results on stably free modules (cf. \cite[Theorem 1]{S1}, \cite[Theorem 7.5]{FRS}). The special case $n=2$ of Theorem 1 was already proven before by Swan and Towber in \cite{SwT} and was also crucially used in seminal work on stably free modules (cf. \cite{F3}).\\
Now unimodular rows of even length play an analogous role for stably trivial symplectic modules as unimodular rows of arbitrary length for stably free modules. Indeed, any unimodular row of length $2n+2$ for some integer $n \geq 0$ has an associated stably trivial symplectic $R$-module of rank $2n$; this symplectic $R$-module is trivial if and only if the corresponding unimodular row of even length can be completed to a symplectic matrix (cf. \cite[Section 3]{F1}). It is therefore natural to ask whether there exists an analogue of Suslin's $n!$-theorem on the completability of unimodular rows of even length to symplectic matrices. In this paper, we prove the following symplectic version of Suslin's $n!$-theorem (cf. Theorem \ref{T1} in the text):\\\\
\textbf{Theorem 1.} Let $R$ be an algebra over a quadratically closed perfect field of characteristic $\neq 2$, $n \geq 0$ an integer and $(a_{0},...,a_{2n+1})$ be any unimodular row of length $2n+2$ over $R$. Furthermore, let $r_{0},...,r_{2n+1} \geq 1$ be integers such that
\begin{itemize}
\item $(2n+1)!$ divides $r_{0}\cdot ... \cdot r_{2n+1}$ and $n$ is even or
\item $2\cdot(2n+1)!$ divides $r_{0}\cdot ... \cdot r_{2n+1}$ and $n$ is odd
\end{itemize}
Then the unimodular row $(a_{0}^{r_{0}},...,a_{2n+1}^{r_{2n+1}})$ is the first row of a symplectic matrix $\varphi \in Sp_{2n+2}(R)$.\\\\
Our proof shows that the divisibility assumptions on the product $r_{0}\cdot ... \cdot r_{2n+1}$ in Theorem 1 cannot be weakened even for algebras over algebraically closed fields of characteristic $\neq 2$ and that the statement of Theorem 1 fails to be true without these assumptions (cf. Remark \ref{Optimal}). Moreover, the methods in the proof of Theorem 1 also enable us to prove similar results for algebras over finite fields of characteristic $\neq 2$ and for algebras over perfect fields containing a square root of $-1$ with characteristic $\neq 2$ under slightly stronger divisibility assumptions (cf. Theorem \ref{T2} in the text):\\\\
\textbf{Theorem 2.} Let $R$ be an algebra over a finite field $k$ of characteristic $\neq 2$ or an algebra over a perfect field $k$ of characteristic $\neq 2$ with $-1 \in {(k^{\times})}^2$, $n \geq 0$ an integer and $(a_{0},...,a_{2n+1})$ be any unimodular row of length $2n+2$ over $R$. Furthermore, let $r_{0},...,r_{2n+1} \geq 1$ be integers such that
\begin{itemize}
\item $2\cdot(2n+1)!$ divides $r_{0}\cdot ... \cdot r_{2n+1}$ and $n$ is even or
\item $4\cdot(2n+1)!$ divides $r_{0}\cdot ... \cdot r_{2n+1}$ and $n$ is odd
\end{itemize}
Then the unimodular row $(a_{0}^{r_{0}},...,a_{2n+1}^{r_{2n+1}})$ is the first row of a symplectic matrix $\varphi \in Sp_{2n+2}(R)$.\\\\
Our proofs make substantial use of $\mathbb{A}^1$-homotopy theory: The main idea is to show that it suffices to prove the statement in the theorem only for the "universal" $k$-algebra $S_{2n+2}$ parametrizing unimodular rows of length $2n+2$ with a chosen section and only "up to $\mathbb{A}^1$-homotopy"; as the algebra $S_{2n+2}$ is smooth over $k$ and its associated scheme has the $\mathbb{A}^1$-homotopy type of a punctured affine space, one may then apply techniques from $\mathbb{A}^1$-homotopy theory involving computations with contracted $\mathbb{A}^1$-homotopy sheaves in order to prove the theorems. We remark that it could still be possible to prove the statement of Theorem 1 for arbitrary algebras over fields of characteristic $\neq 2$ or even for arbitrary commutative rings in the future; however, this is currently out of reach and completely open.\\
Unsurprisingly, Theorem 1 has immediate considerable applications: As a direct consequence of Theorem 1, we can prove that every unimodular row of length $d+1$ over a reduced affine algebra $R$ of odd dimension $d \geq 3$ over an algebraically closed field of characteristic $\neq 2$ is the first row of a symplectic matrix over $R$ (cf. Theorem \ref{T3}); furthermore, we prove that any unimodular row of length $d$ over a normal affine algebra of even dimension $d \geq 4$ over an algebraically closed field $k$ with $(d-1)!\in k^{\times}$ is the first row of a symplectic matrix (cf. Theorem \ref{T4}). These results can be considered symplectic versions of the main results in \cite{S1} and \cite{FRS} and were proven with a lot of technical efforts for smooth affine algebras in \cite{Sy2}; Theorem 1 allows us to drop the smoothness assumption and to give very simple proofs.\\
Finally, Theorem 1 also allows us to prove results on stably free modules over affine algebras over algebraically closed fields: Recall that classical results of Bass and Suslin imply that stably free modules of rank $\geq d$ over affine algebras of dimension $d$ over algebraically closed fields are always free (cf. \cite[Chapter IV, Theorem 3.4]{HB}, \cite[Theorem 1]{S1}). The main result in \cite{FRS} shows that stably free modules of rank $d-1$ over normal affine algebras of dimension $d \geq 4$ over an algebraically closed field $k$ with $(d-1)! \in k^{\times}$ are always free; the same statement is proven for smooth $k$-algebras of dimension $3$ (cf. \cite[Theorem 7.5]{FRS}). Theorem 1 enables us deduce the following criterion for general (not necessarily smooth) affine algebras of dimension $3$ (cf. Theorem \ref{T5}):\\\\
\textbf{Theorem 3.} Let $R$ be an affine algebra of dimension $3$ over an algebraically closed field $k$ with characteristic $\neq 2$. Then all stably free $R$-modules of rank $2$ are free if and only if $W_{SL} (R) = 0$.\\\\
The abelian group $W_{SL}(R)$ is a Hermitian $K$-theory group and was introduced in \cite[\S 3]{SV}. Theorem 3 provides a precise cohomological criterion for all stably free modules of rank $2$ over $3$-dimensional affine algebras to be free and is the first of its kind in this generality. In case of a smooth affine algebra of dimension $3$, stably free modules of rank $2$ are always free by \cite[Corollary 6.8]{AF} and therefore the group $W_{SL}(R)$ has to be trivial in this situation; the question whether the same holds for non-smooth algebras is completely open.\\
While it was well-known for a long time that stably free modules of rank $d-2$ over smooth affine algebras of dimension $d \geq 4$ over algebraically closed fields need not be free (cf. \cite{NMK}), we can prove the following precise cohomological criterion for all stably free modules of rank $2$ over a normal affine algebra of dimension $4$ to be free (cf. Theorem \ref{T6}):\\\\
\textbf{Theorem 4.} Let $R$ be a normal affine algebra of dimension $4$ over an algebraically closed field $k$ with $6 \in k^{\times}$. Then all stably free $R$-modules of rank $2$ are free if and only if $W_{SL} (R) = 0$.\\\\
This is yet another consequence of Theorem 1 and generalizes the main result of \cite{Sy1}, where the same statement was proven for smooth affine algebras of dimension $4$ over an algebraically closed field $k$ with $6 \in k^{\times}$. Again, Theorem 4 is the first result of its kind in this generality.\\
The paper is structured as follows: We recall basic definitions and facts about unimodular rows in Section \ref{2.1}. Then we give a brief introduction to motivic homotopy theory as needed for this paper in Section \ref{2.2}. We prove the main results of this paper in Section \ref{3}. In Section \ref{4}, we discuss several applications of our main results.

\subsection*{Acknowledgements}
The author would like to thank Aravind Asok, Jean Fasel, Samuel Lerbet and Keyao Peng for helpful comments. The author was funded by the Deutsche Forschungsgemeinschaft (DFG, German Research Foundation) - Project numbers 461453992 and 544731044.

\section{Preliminaries}\label{Preliminaries}\label{2}

\subsection{Unimodular rows}\label{Unimodular rows}\label{2.1}

Let $R$ be a commutative ring. For any integer $n \geq 1$, a unimodular row of length $n$ over $R$ is a row vector $(a_{1},...,a_{n})$ of length $n$ such that $a_{i} \in R$, $1 \leq i \leq n$, and $\langle a_{1},...,a_{n} \rangle = R$. We denote by $Um_n (R)$ the set of unimodular row vectors of length $n$ over $R$. By definition, if $a=(a_{1},...,a_{n}) \in Um_n (R)$, then there exists a row vector $b=(b_{1},...,b_{n})$ with $b_i \in R$, $1 \leq i \leq n$, such that $a b^t = \sum_{i=1}^{n} a_{i}b_{i}=1$; we call any such row vector $b$ a section of $a$. Clearly, the group $GL_n (R)$ of invertible $n \times n$-matrices over $R$ acts on the right on $Um_n (R)$ by matrix multiplication. In particular, the subgroup $SL_n (R)$ of matrices with determinant $1$ and the subgroup $E_n (R)$ generated by elementary matrices act on the right on $Um_n (R)$ by restriction. If $n$ is even, then the subgroup $Sp_{n}(R)$ of symplectic matrices and its subgroup $ESp_{n}(R)$ generated by elementary symplectic matrices act on the right on $Um_n (R)$ by restriction as well.

\subsection{Motivic homotopy theory}\label{Motivic homotopy theory}\label{2.2}

In this section, we give a short introduction to motivic homotopy theory as needed for this paper; our main reference is \cite{MV}. For this purpose, let $k$ be a fixed perfect field.\\
We denote by $Sm_k$ the category of smooth separated schemes of finite type over $Spec(k)$ and by $Spc_k$ the category of spaces, i.e., the category of simplicial Nisnevich sheaves on $Sm_k$. Similarly, we denote by $Spc_{k,\ast}$ the category of pointed spaces, i.e., the category of pointed simplicial Nisnevich sheaves on $Sm_k$. We will refer to objects of $Spc_k$ (resp. $Spc_{k,\ast}$) as spaces (resp. pointed spaces). Note that any (pointed) simplicial set and also any (pointed) smooth $k$-scheme $X \in Sm_k$ define a (pointed) space.\\
We denote by $\mathcal{H}(k)$ the unstable $\mathbb{A}^1$-homotopy category over $k$, which is the homotopy category of a model structure on $Spc_k$; the weak equivalences of this model structure are called $\mathbb{A}^1$-weak equivalences. Analogously, we denote by $\mathcal{H}_{\ast}(k)$ the pointed unstable $\mathbb{A}^1$-homotopy category over $k$, which is the homotopy category of the corresponding model structure on $Spc_{k,\ast}$; the weak equivalences of this model structure are called pointed $\mathbb{A}^1$-weak equivalences. We refer the reader to \cite{MV} for details and to \cite{Ho} for background on model categories.\\
If $\mathcal{X},\mathcal{Y}$ are spaces, the set of morphisms from $\mathcal{X}$ to $\mathcal{Y}$ in $\mathcal{H}(k)$ will be denoted by $[\mathcal{X},\mathcal{Y}]_{\mathbb{A}^1}$. Analogously, if $(\mathcal{X},x), (\mathcal{Y},y)$ are pointed spaces, the set of morphisms from $\mathcal{X}$ to $\mathcal{Y}$ in $\mathcal{H}_{\ast}(k)$ will be denoted by $[(\mathcal{X},x),(\mathcal{Y},y)]_{\mathbb{A}^{1},\ast}$.\\
There is a forgetful functor $f: \mathcal{H}_{\ast}(k) \rightarrow \mathcal{H}(k), (\mathcal{X},x) \mapsto \mathcal{X}$. As in topology, one can define a smash product $(\mathcal{X},x) \wedge (\mathcal{Y},y)$ of two pointed spaces $(\mathcal{X},x),(\mathcal{Y},y) \in Spc_{k,\ast}$. For any pointed space $(\mathcal{X},x) \in Spc_{k,\ast}$, one obtains a functor $(\mathcal{X},x) \wedge -: \mathcal{H}_{\ast}(k) \rightarrow \mathcal{H}_{\ast}(k)$. If $(\mathcal{X},x) = (S^1, \ast)$ is the simplicial $1$-sphere with canonical basepoint, then this functor is called simplicial suspension and denoted $\Sigma_{s}^{1} = (S^1,\ast)\wedge -: \mathcal{H}_{\ast}(k) \rightarrow \mathcal{H}_{\ast}(k)$; we usually omit the basepoint in the notation. Similarly, we usually view $\mathbb{G}_m$ as a pointed space with canonical basepoint $1$.\\
For a pointed space $(\mathcal{X},x)$ and integers $i,j \geq 0$, we define $\pi_{i,j}^{\mathbb{A}^1}(\mathcal{X},x)$ as the Nisnevich sheaf associated to the presheaf $U \mapsto [{S^1}^{\wedge i} \wedge \mathbb{G}_{m}^{\wedge j} \wedge U_{+}, (\mathcal{X},x)]_{\mathbb{A}^1,\ast}$ on $Sm_k$. The sheaves $\pi_{i}^{\mathbb{A}^1}(\mathcal{X},x):=\pi_{i,0}^{\mathbb{A}^1}(\mathcal{X},x)$ are called $\mathbb{A}^1$-homotopy sheaves of the pointed space $(\mathcal{X},x)$; these sheaves are sheaves of sets for $i=0$, of groups for $i=1$ and of abelian groups for $i \geq 2$.\\
We say that a sheaf of abelian groups $\textbf{A}$ on $Sm_k$ is strictly $\mathbb{A}^1$-invariant if the map $H^{i}_{Nis}(X,\textbf{A}) \rightarrow H^{i}_{Nis}(X \times \mathbb{A}^1,\textbf{A})$ induced by the projection $X \times \mathbb{A}^1 \rightarrow X$ is a bijection for all $X \in Sm_k$ and $i \geq 0$. A theorem of Morel asserts that the $\mathbb{A}^1$-homotopy sheaves $\pi_{i}^{\mathbb{A}^1}(\mathcal{X},x)$ of a pointed space $(\mathcal{X},x)$ are strictly $\mathbb{A}^1$-invariant if $i \geq 2$ (cf. \cite[Corollary 5.2]{Mo}). The category of strictly $\mathbb{A}^1$-invariant sheaves of abelian groups ${Ab}_{\mathbb{A}^1}(k)$ is abelian (cf. \cite[Corollary 5.24]{Mo}).\\
There is an exact functor ${()}_{-1}: {Ab}_{\mathbb{A}^1}(k) \rightarrow {Ab}_{\mathbb{A}^1}(k), \textbf{A} \mapsto \textbf{A}_{-1}$ called contraction functor (cf. \cite[Lemma 6.33]{Mo}). For an integer $n \geq 1$, we denote by ${()}_{-n}: {Ab}_{\mathbb{A}^1}(k) \rightarrow {Ab}_{\mathbb{A}^1}(k), \textbf{A} \mapsto \textbf{A}_{-n}$ the iteration of this contraction functor. There are canonical isomorphisms
\begin{center}
$\pi_{i,j}^{\mathbb{A}^1}(\mathcal{X},x) \cong {\pi_{i}^{\mathbb{A}^1}(\mathcal{X},x)}_{-j}$
\end{center}
for $i \geq 2, j \geq 0$ and any pointed space $(\mathcal{X},x)$ with $\pi^{\mathbb{A}^1}_{0}(\mathcal{X},x)=\ast$ (cf. \cite[Theorem 5.13]{Mo}).\\
For $n \in \mathbb{Z}$, we denote by $\textbf{K}^{MW}_{n}$ the $n$-th unramified Milnor-Witt $K$-theory sheaf, by $\textbf{K}^{M}_{n}$ the $n$-th unramified Milnor $K$-theory sheaf and by $\textbf{I}^{n}$ the unramified sheaf of the $n$-th power of the fundamental ideal described in \cite[\S 2]{Mo}. These sheaves are strictly $\mathbb{A}^1$-invariant and linked by a canonical short exact sequence
\begin{center}
$0 \rightarrow \textbf{I}^{n+1} \rightarrow \textbf{K}^{MW}_{n} \rightarrow \textbf{K}^{M}_{n} \rightarrow 0$.
\end{center}
of strictly $\mathbb{A}^1$-invariant sheaves. The contractions of these sheaves are given as ${(\textbf{K}^{MW}_{n})}_{-1}=\textbf{K}^{MW}_{n-1}$, ${(\textbf{K}^{M}_{n})}_{-1}=\textbf{K}^{M}_{n-1}$ and ${(\textbf{I}^{n})}_{-1}=\textbf{I}^{n-1}$ and the canonical short exact sequence is compatible with these identifications.\\
The group $\textbf{K}^{MW}_{0}(k)$ can be identified with the Grothendieck-Witt group $GW(k)$ of non-degenerate symmetric bilinear forms over $k$; if we let $I(k)$ be the fundamental ideal of $k$, i.e., the kernel of the rank homomorphism $W(k) \rightarrow \mathbb{Z}/2\mathbb{Z}$ modulo $2$ from the Witt ring of non-degenerate symmetric bilinear forms over $k$ to $\mathbb{Z}/2\mathbb{Z}$, then, for $n \geq 1$, the groups $\textbf{I}^{n}(k)$ can be identified with $n$-th powers $I^n(k)$ the fundamental ideal $I(k)$.\\
Now let $n \geq 1$ be an integer and let $Q_{2n-1}$ be the smooth affine scheme $Spec (k[x_{1},...,x_{n},y_{1},...,y_{n}]/\langle \sum_{i=1}^{n} x_{i}y_{i} - 1 \rangle) \in Sm_k$. It is well-known that the projection morphism $p_{2n-1}:Q_{2n-1} \rightarrow \mathbb{A}^{n}\setminus 0$ on the coefficients $x_{1},...,x_{n}$ is an $\mathbb{A}^{1}$-weak equivalence. In particular, if we equip $\mathbb{A}^{n}\setminus 0$ with $(1,0,..,0)$ and $Q_{2n-1}$ with $(1,0,..,0,1,0,..,0)$ as basepoints, we obtain a pointed $\mathbb{A}^{1}$-weak equivalence

\begin{center}
$Q_{2n-1} \simeq_{\mathbb{A}^{1}} \mathbb{A}^{n}\setminus 0$.
\end{center}

As indicated, we usually omit these canonical basepoints in the notation. Now let $R$ be any $k$-algebra, $X = Spec(R)$ and $Sch_k$ be the category of schemes over $Spec(k)$. Then there is a canonical bijection

\begin{center}
$Um_{n} (R) \cong Hom_{Sch_{k}} (X, \mathbb{A}^{n}\setminus 0)$.
\end{center}

In other words, unimodular rows of length $n$ over $R$ correspond exactly to morphisms $X \rightarrow \mathbb{A}^{n}\setminus 0$ of schemes over $Spec(k)$. Similarly, it follows easily that there is a canonical bijection

\begin{center}
$\{(a,b)|a,b \in Um_{n} (R), a b^{t} = 1\} = Hom_{Sch_{k}} (X, Q_{2n-1})$,
\end{center}

So unimodular rows of length $n$ over $R$ with a chosen section correspond exactly to morphisms $X \rightarrow Q_{2n-1}$ of schemes over $Spec(k)$.\\
Now assume furthermore that $R$ is smooth, $k$ perfect with $char(k) \neq 2$ and $n \geq 3$. As a direct consequence of \cite[Remark 7.10]{Mo} and \cite[Theorem 2.1]{F2}, one concludes that the bijection above descends to a bijection

\begin{center}
$\mathit{Um}_{n} (R)/E_{n} (R) \cong [X, \mathbb{A}^{n}\setminus 0]_{\mathbb{A}^{1}}$.
\end{center}

It is well-known that there are pointed $\mathbb{A}^1$-weak equivalences

\begin{center}
$\mathbb{A}^{n}\setminus 0 \simeq_{\mathbb{A}^{1}} \Sigma_{s}^{n-1} \mathbb{G}_{m}^{\wedge n}$,
\end{center}

for $n \geq 1$, where $\mathbb{G}_m$ has $1$ as a canonical basepoint.  For $n \geq 3$, by \cite[Theorem 5.40]{Mo}, one has isomorphisms
\begin{center}
\begin{equation*}
\pi_{i}^{\mathbb{A}^1}(\mathbb{A}^{n}\setminus 0,\ast) \cong
\begin{cases}
0 & \text{if } i \leq n-2 \\
\textbf{K}^{MW}_{n} & \text{if }i=n-1. \\
\end{cases}
\end{equation*}
\end{center}
In particular, by \cite[Corollary 5.43]{Mo}, there is a canonical group isomorphism
\begin{center}
$\pi_{n-1,n}^{\mathbb{A}^1}(\mathbb{A}^{n}\setminus 0)(k) =[\mathbb{A}^{n} \setminus 0, \mathbb{A}^{n} \setminus 0]_{\mathbb{A}^{1},\ast} \cong \textbf{K}^{MW}_{0}(k) = GW (k)$
\end{center}
called the motivic Brouwer degree. It follows directly from \cite[Proposition 2.1.9]{AFH} that the morphism $\Psi^{r}: \mathbb{A}^n \setminus 0 \rightarrow \mathbb{A}^n \setminus 0, (x_{1},...,x_{n}) \mapsto (x_{1},...,x_{n}^r)$ corresponds to
\begin{center}
$r_{\epsilon} = \sum_{i=1}^{r} \langle {(-1)}^{i-1} \rangle  \in GW (k)$
\end{center}
under the identification $[\mathbb{A}^{n} \setminus 0, \mathbb{A}^{n} \setminus 0]_{\mathbb{A}^{1},\ast} \cong GW (k)$ above.

\section{Proof of the main result}\label{Proof of the main result}\label{3}

In this section we prove the main results of this paper:

\begin{Thm}\label{T1}
Let $R$ be an algebra over a quadratically closed perfect field of characteristic $\neq 2$, $n \geq 0$ an integer and $(a_{0},...,a_{2n+1})$ be any unimodular row of length $2n+2$ over $R$. Furthermore, let $r_{0},...,r_{2n+1} \geq 1$ be integers such that
\begin{itemize}
\item $(2n+1)!$ divides $r_{0}\cdot ... \cdot r_{2n+1}$ and $n$ is even or
\item $2\cdot(2n+1)!$ divides $r_{0}\cdot ... \cdot r_{2n+1}$ and $n$ is odd
\end{itemize}
Then the unimodular row $(a_{0}^{r_{0}},...,a_{2n+1}^{r_{2n+1}})$ is the first row of a symplectic matrix $\varphi \in Sp_{2n+2}(R)$.
\end{Thm}

\begin{proof}
There is nothing to prove if $n=0$: Any unimodular row of length $2$ is the first row of a matrix with determinant $1$ and any $2 \times 2$-matrix with determinant $1$ is symplectic. So let $n \geq 1$.\\
Furthermore, we only have to prove that the row $(a_{0},...,a_{2n-1}^{r_{0}\cdot ... \cdot r_{2n+1}})$ is the first row of a symplectic matrix: By \cite[Theorem]{V}, there is an elementary matrix $E$ such that $(a_{0},...,a_{2n-1}^{r_{0}\cdot ... \cdot r_{2n+1}})E=(a_{0}^{r_{0}},...,a_{2n-1}^{r_{2n+1}})$. Then it follows from \cite[Theorem 3.9]{G} that there is also a matrix $E' \in ESp_{2n}(R) \subset Sp_{2n}(R)$ such that $(a_{0},...,a_{2n-1}^{r_{0}\cdot ... \cdot r_{2n+1}})E'=(a_{0}^{r_{0}},...,a_{2n-1}^{r_{2n+1}})$. In particular, once we prove that $(a_{0},...,a_{2n-1}^{r_{0}\cdot ... \cdot r_{2n+1}})$ is the first row of a symplectic matrix, the same holds for the row $(a_{0}^{r_{0}},...,a_{2n-1}^{r_{2n+1}})$.\\
Moreover, let $S_{4n+3}= k[x_{1},...,x_{2n+2},y_{1},...,y_{2n+2}]/\langle \sum_{i=1}^{2n+2} x_{i}y_{i} - 1 \rangle$ and note that $Spec(S_{4n+3}) = Q_{4n+3}$ by definition. Then it suffices to prove the statement in the theorem only for the unimodular row $(x_{0},...,x^{r_{0}\cdot ... \cdot r_{2n+1}})$ of length $2n+2$ over $S_{4n+3}$: Indeed, the unimodular row $a=(a_{0},...,a_{2n+1})$ over $R$ and some chosen section $b=(b_{0},...,b_{2n+1})$ correspond to a homomorphism
\begin{center}
$\varphi_{a,b}: S_{4n+3} \rightarrow R$
\end{center}
of $k$-algebras. This homomorphism induces a homomorphism
\begin{center}
$\varphi^{\ast}_{a,b}: Sp_{2n+2}(S_{4n+3}) \rightarrow Sp_{2n+2}(R)$.
\end{center}
If $M \in Sp_{2n+2}(S_{4n+3})$ is a symplectic matrix with first row $(x_{1},...,x_{2n+2}^{r_{0}\cdot ... \cdot r_{2n+1}})$, then $\varphi^{\ast}_{a,b}(M)$ is clearly a symplectic matrix with first row $(a_{0},...,a_{2n+1}^{r_{0}\cdot ... \cdot r_{2n+1}})$, as desired.\\
So let us now prove the statement of the theorem for the unimodular row $x=(x_{1},...,x_{2n+2}^{r_{0}\cdot ... \cdot r_{2n+1}})$ of length $2n+2$ over $S_{4n+3}$. For this, we choose some section $y$ of $x$. This choice determines a morphism
\begin{center}
$\varphi_{x,y}: Q_{4n+3} \rightarrow Q_{4n+3}$
\end{center}
of schemes over $Spec(k)$. Furthermore, we have a morphism
\begin{center}
$\pi: Sp_{2n+2} \rightarrow A^{2n+2}\setminus 0$
\end{center}
of schemes over $Spec(k)$ corresponding to the projection onto the first row of a symplectic matrix. It now suffices to find a lift in $\mathcal{H}(k)$ in the diagram
\begin{center}
$\begin{xy}
  \xymatrix{
     & & Sp_{2n+2} \ar[d]^{\pi} \\
     Q_{4n+3} \ar[r]_{\varphi_{x,y}} \ar@{.>}[rru] & Q_{4n+3} \ar[r] & \mathbb{A}^{2n+2}\setminus 0}
\end{xy}$
\end{center}
where the morphism $Q_{4n+3} \rightarrow \mathbb{A}^{2n+2}\setminus 0$ is the projection onto the coordinates $x_{1},...,x_{2n+2}$: Indeed, the space $Sp_{2n+2}$ is $\mathbb{A}^1$-naive by \cite[Theorem 4.2.12]{AHW}, so any morphism $Q_{4n+3} \rightarrow Sp_{4n+3}$ in $\mathcal{H}(k)$ comes from an actual morphism $Q_{4n+3} \rightarrow Sp_{4n+3}$ of schemes over $Spec(k)$. If there is a lift in diagram above in the category $\mathcal{H}(k)$, this means that, in particular, there exist a symplectic matrix $M \in Sp_{2n+2}(S_{4n+3})$ whose first row equals $(x_{1},...,x_{2n+2}^{r_{0},...,r_{2n+1}})$ in
\begin{center}
$[Q_{4n+3},\mathbb{A}^{2n+2}\setminus]_{\mathbb{A}^1} \cong Um_{2n+2}(S_{4n+3})/E_{2n+2}(S_{4n+3})$.
\end{center}
In other words, the first row of $M$ is $x E$ for some $E_{2n+2}(S_{4n+3})$. But since $x E_{2n+2}(S_{4n+3}) = x ESp_{2n+2}(S_{4n+3})$ by \cite[Theorem 3.9]{G}, there is a matrix $E' \in ESp_{2n+2}(S_{4n+3})$ with $x E = x E'$; then, by construction, the first row of $M {E'}^{-1} \in Sp_{2n+2}(S_{4n+3})$ is precisely $x$, as desired. So it indeed suffices to find a lift in $\mathcal{H}(k)$ in the diagram above.\\
Now we realize that it actually suffices to find the desired lift in $\mathcal{H}_{\ast}(k)$: All the morphisms in the diagram are actually pointed, when we equip $Q_{4n+3}$ and $\mathbb{A}^{2n+2}\setminus 0$ with their canonical basepoints and $Sp_{2n+2}$ with the identity matrix as its basepoint; once we have found a lift in $\mathcal{H}_{\ast}(k)$, applying the forgetful functor yields a lift in $\mathcal{H}(k)$.\\
By using the canonical pointed $\mathbb{A}^1$-weak equivalence $Q_{4n+3} \simeq_{\mathbb{A}^1} \mathbb{A}^{2n+2} \setminus 0$, we may equivalently look for a lift in the diagram
\begin{center}
$\begin{xy}
\xymatrix{
& Sp_{2n+2} \ar[d]^{\pi} \\
\mathbb{A}^{2n+2}\setminus 0 \ar[r]_{\Psi^{r}} \ar@{.>}[ru] & \mathbb{A}^{2n+2}\setminus 0}
\end{xy}$
\end{center}
in the category $\mathcal{H}_{\ast}(k)$, where $\Psi^{r}: \mathbb{A}^{2n+2}\setminus 0 \rightarrow \mathbb{A}^{2n+2}\setminus 0, (x_{1},...,x_{2n+2}) \mapsto (x_{1},...,x_{2n+2}^{r})$ with $r= r_{0}\cdot ... \cdot r_{2n+1}$. Therefore it is left to show that $\Psi^{r}$ lies in the image of the map
\begin{center}
$\pi_{\ast}: [A^{2n+2}\setminus 0, Sp_{2n+2}]_{\mathbb{A}^{1},\ast} \rightarrow [A^{2n+2}\setminus 0, A^{2n+2}\setminus 0]_{\mathbb{A}^{1},\ast}$.
\end{center}
We can identify this map with the map
\begin{center}
$\pi^{\mathbb{A}^1}_{2n+1,2n+2}(\pi): \pi_{2n+1,2n+2}^{\mathbb{A}^1}(Sp_{2n+2})(k) \rightarrow \pi_{2n+1,2n+2}^{\mathbb{A}^1}(A^{2n+2}\setminus 0)(k)\cong \textbf{K}^{MW}_{0}(k)$
\end{center}
and $\Psi^{r}$ with the element $r_{\epsilon} \in GW(k)$.\\
Following \cite[Section 3]{AF}, we first define
\begin{center}
$T'_{2n+2}:= coker (\pi^{\mathbb{A}^1}_{2n+1}(Sp_{2n+2}) \xrightarrow{\pi^{\mathbb{A}^1}_{2n+1}(\pi)} \pi^{\mathbb{A}^1}_{2n+1}(\mathbb{A}^{2n+2}\setminus 0))$.
\end{center}
The sheaf $\pi^{\mathbb{A}^1}_{2n+1}(\mathbb{A}^{2n+2}\setminus 0)$ is canonically isomorphic to the Milnor-Witt $K$-theory sheaf $\textbf{K}^{MW}_{2n+2}$. Composing $\pi^{\mathbb{A}^1}_{2n+1}(\pi)$ with the canonical epimorphism $\textbf{K}^{MW}_{2n+2} \rightarrow \textbf{K}^{M}_{2n+2}$, we obtain a new morphism
\begin{center}
$\pi^{\mathbb{A}^1}_{2n+1}(Sp_{2n+2}) \xrightarrow{\pi^{\mathbb{A}^1}_{2n+1}(\pi)'} \textbf{K}^{M}_{2n+2}$.
\end{center}
Again following \cite[Section 3]{AF}, we define
\begin{center}
$S'_{2n+2}:= coker (\pi^{\mathbb{A}^1}_{2n+1}(Sp_{2n+2}) \xrightarrow{\pi^{\mathbb{A}^1}_{2n+1}(\pi)'} \textbf{K}^{M}_{2n+2})$.
\end{center}
Then it follows from \cite[Lemma 3.1]{AF} that the canonical exact sequence of strictly $\mathbb{A}^1$-invariant sheaves
\begin{center}
$0 \rightarrow \textbf{I}^{2n+3} \rightarrow \textbf{K}^{MW}_{2n+2} \rightarrow \textbf{K}^{M}_{2n+2} \rightarrow 0$
\end{center}
induces an exact sequence of strictly $\mathbb{A}^1$-invariant sheaves of the form
\begin{center}
$\textbf{I}^{2n+3} \rightarrow T'_{2n+2} \rightarrow S'_{2n+2} \rightarrow 0$.
\end{center}
Contracting this exact sequence $2n+2$ times and evaluating at the base field $k$, we obtain an exact sequence of abelian groups of the form
\begin{center}
$I(k) \rightarrow {(T'_{2n+2})}_{-(2n+2)}(k) \rightarrow {(S'_{2n+2})}_{-(2n+2)}(k) \rightarrow 0$,
\end{center}
where ${(T'_{2n+2})}_{-(2n+2)}(k) \rightarrow {(S'_{2n+2})}_{-(2n+2)}(k)$ is induced by the canonical epimorphism $GW(k) \rightarrow \textbf{K}_{0}(k) \cong \mathbb{Z}$ (i.e., the rank homomorphism).\\
Now it follows directly from \cite[Lemmas 7.1 and 7.2]{AF} that
\begin{center}
\begin{equation*}
{(S'_{2n+2})}_{-(2n+2)}(k) \cong
\begin{cases}
\mathbb{Z}/(2n+1)!\mathbb{Z} & \text{if n even} \\
\mathbb{Z}/2(2n+1)!\mathbb{Z} & \text{if n odd} \\
\end{cases}
\end{equation*}
\end{center}
as quotients of $\textbf{K}_{0}(k) \cong \mathbb{Z}$. In particular, since $r = r_{0}\cdot ... \cdot r_{2n+1}$ is divisible by $(2n+1)!$ if $n$ is even and by $2(2n+1)!$ if $n$ is odd by assumption, the class of $r_{\epsilon}$ in ${(T'_{2n+2})}_{-(2n+2)}(k)$ is mapped to $0$ in ${(S'_{2n+2})}_{-(2n+2)}(k)$.\\
Finally, since $k$ is quadratically closed, the fundamental ideal $I(k)$ over $k$ is trivial, i.e., $I(k) = 0$, and hence the epimorphism
\begin{center}
${(T'_{2n+2})}_{-(2n+2)}(k) \rightarrow {(S'_{2n+2})}_{-(2n+2)}(k)$
\end{center}
is in fact an isomorphism. In particular, the class of $r_{\epsilon}$ in
\begin{center}
${(T'_{2n+2})}_{-(2n+2)}(k)= coker (\pi^{\mathbb{A}^1}_{2n+1,2n+2}(\pi))$
\end{center}
is equal to $0$; in other words, $r_{\epsilon} \in GW(k) \cong \pi_{2n+1,2n+2}(\mathbb{A}^{2n+2}\setminus 0)$ lies in the image of the map $\pi_{2n+1,2n+2}^{\mathbb{A}^1}(\pi)$, as desired. This finishes the proof.
\end{proof}

\begin{Thm}\label{T2}
Let $R$ be an algebra over a finite field $k$ of characteristic $\neq 2$ or an algebra over a perfect field $k$ of characteristic $\neq 2$ with $-1 \in {(k^{\times})}^2$, $n \geq 0$ an integer and $(a_{0},...,a_{2n+1})$ be any unimodular row of length $2n+2$ over $R$. Furthermore, let $r_{0},...,r_{2n+1} \geq 1$ be integers such that
\begin{itemize}
\item $2\cdot(2n+1)!$ divides $r_{0}\cdot ... \cdot r_{2n+1}$ and $n$ is even or
\item $4\cdot(2n+1)!$ divides $r_{0}\cdot ... \cdot r_{2n+1}$ and $n$ is odd
\end{itemize}
Then the unimodular row $(a_{0}^{r_{0}},...,a_{2n+1}^{r_{2n+1}})$ is the first row of a symplectic matrix $\varphi \in Sp_{2n+2}(R)$.
\end{Thm}

\begin{proof}
All steps in the proof of Theorem \ref{T1} except the last paragraph work over any perfect field with characteristic $\neq 2$. We use the assumption that $k$ is quadratically closed only in the last paragraph in order to conclude that $I(k)=0$. Now it is well-known (cf. \cite[Chapter II, \S 3]{Sch}) that if $k$ is a finite field of characteristic $\neq 2$, then
\begin{center}
\begin{equation*}
W(k) \cong
\begin{cases}
\mathbb{Z}/2\mathbb{Z}\times \mathbb{Z}/2\mathbb{Z} & \text{if } |k| \equiv 1~mod~4\\
\mathbb{Z}/4\mathbb{Z} & \text{if } |k| \equiv 3~mod~4\\
\end{cases}
\end{equation*}
\end{center}
In the second case, the rank homomorphism $W(k) \rightarrow \mathbb{Z}/2\mathbb{Z}$ modulo $2$ then corresponds to the projection $\mathbb{Z}/4\mathbb{Z} \rightarrow \mathbb{Z}/2\mathbb{Z}$. Furthermore, if $k$ is any perfect field of characteristic $\neq 2$ with $-1 \in {(k^{\times})}^{2}$, then the element $h=\langle 1\rangle + \langle -1 \rangle$ of the Grothendieck-Witt ring $GW(k)$ defined by the hyperbolic form will be equal to $\langle 1 \rangle + \langle 1 \rangle$ and therefore the Witt ring $W(k)$ will automatically be $2$-torsion as it is the quotient of $GW(k)$ by the ideal generated by $h$.\\
It follows directly from the observations in the previous paragraph that in all relevant cases $I(k)$ is $2$-torsion. In particular, the kernel of the map
\begin{center}
${(T'_{2n+2})}_{-(2n+2)}(k) \rightarrow {(S'_{2n+2})}_{-(2n+2)}(k)$
\end{center}
is $2$-torsion as it is a quotient of $I(k)$.\\
If we then use the stronger assumption that $r = r_{0}\cdot ... \cdot r_{2n+1}$ is divisible by $2 \cdot (2n+1)!$ if $n$ is even and $4 \cdot (2n+1)!$ if $n$ is odd, then the class of $r_{\epsilon} = 2 \cdot {(\dfrac{r}{2})}_{\epsilon}$ in ${(T'_{2n+2})}_{-(2n+2)}(k)$ will be a $2$-fold multiple of an element of the kernel of the map
\begin{center}
${(T'_{2n+2})}_{-(2n+2)}(k) \rightarrow {(S'_{2n+2})}_{-(2n+2)}(k)$
\end{center}
and hence $0$ in ${(T'_{2n+2})}_{-(2n+2)}(k)$ as this kernel is $2$-torsion. This finishes the proof.
\end{proof}

\begin{Rem}\label{Optimal}
The proof of Theorem \ref{T1} also shows that the unimodular row $(x_{1},...,x_{2n+2}^{r_{0}\cdot ... \cdot r_{2n+1}})$ of length $2n+2$ over $S_{4n+3}$ cannot be the first row of a symplectic matrix if $r=r_{0}\cdot ... \cdot r_{2n+1}$ is not divisible by $(2n+1)!$ and $n$ is even or if $r$ is not divisible by $2\cdot (2n+1)!$ and $n$ is odd. Indeed, in these cases, the class of $r_{\epsilon}$ in ${(T'_{2n+2})}_{-(2n+2)}(k)$ cannot be $0$ as it is not mapped to $0$ in ${(S'_{2n+2})}_{-(2n+2)}(k)$. Hence the divisibility assumptions in Theorem \ref{T1} cannot be weakened. However, it might still be possible to prove the statement of Theorem \ref{T1} for arbitrary algebras over fields or even for arbitrary commutative rings.
\end{Rem}

\section{Applications}\label{4}

In this final section, we discuss several applications of the main results of this paper. As a first application, we prove symplectic versions of the celebrated results on stably free modules in \cite{S1} and \cite{FRS}. Note that some symplectic versions were already proven for smooth affine algebras in \cite{Sy2}, but Theorem \ref{T1} enables us to prove these results for affine algebras which are not necessarily smooth over $k$:

\begin{Thm}\label{T3}
Let $R$ be a reduced affine algebra of odd dimension $d \geq 3$ over an algebraically closed field $k$ of characteristic $\neq 2$. Then $Sp_{d+1}(R)$ acts transitively on $Um_{d+1}(R)$.
\end{Thm}

\begin{proof}
The proof of \cite[Theorem 1]{S2} shows that any unimodular row $(a_{1},...,a_{d+1})$ of length $d+1$ over $R$ can be transformed via elementary matrices to a unimodular row of the form $(b_{1},...,b_{d+1}^{2 \cdot d!})$ (see also \cite[Theorem 3.1]{Sy2}). By \cite[Theorem 3.9]{G}, $(a_{1},...,a_{d+1})$ can also be transformed via a matrix $E \in ESp_{d+1}(R)$ to $(b_{1},...,b_{d+1}^{2 \cdot d!})$, i.e., $(a_{1},...,a_{d+1})E=(b_{1},...,b_{d+1}^{2 \cdot d!})$. By Theorem \ref{T1}, there is a symplectic matrix $M \in Sp_{d+1}(R)$ with first row $(b_{1},...,b_{d+1}^{2 \cdot d!})$. Then $M E^{-1}$ has first row $(a_{1},...,a_{d+1})$.
\end{proof}



\begin{Thm}\label{T4}
Let $R$ be a normal affine algebra of even dimension $d \geq 4$ over an algebraically closed field $k$ with $(d-1)! \in k^{\times}$. Then $Sp_{d}(R)$ acts transitively on $Um_{d}(R)$.
\end{Thm}

\begin{proof}
The proof of \cite[Theorem 7.5]{FRS} shows that any unimodular row $(a_{1},...,a_{d})$ of length $d$ over $R$ can be transformed via elementary matrices to a unimodular row of the form $(b_{1},...,b_{d}^{2 \cdot (d-1)!})$ (see also \cite[Theorem 3.3]{Sy2}). Hence we can repeat the reasoning of the proof of Theorem \ref{T3}.
\end{proof}

Finally, recall that the abelian group $W_{SL} (R)$ of a commutative ring $R$ was defined in \cite[\S 3]{SV}. A brief introduction to this group can also be found in \cite[Section 2.B]{Sy1}. For the purpose of this paper, we just recall some basic facts about the group $W_{SL}(R)$: For $n \in \mathbb{N}$ and any commutative ring $R$, let $A_{2n}(R)$ denote the set of invertible alternating matrices of rank $2n$ over $R$ with Pfaffian $1$. One has embeddings $A_{2n}(R) \rightarrow A_{2n+2m}(R), M \mapsto M \perp \psi_{2m}$ for $m,n \in \mathbb{N}$ and the abelian group $W_{SL}(R)$ can then be defined as the set of equivalence classes
\begin{center}
$W_{SL}(R) := \bigcup_{n \in \mathbb{N}} A_{2n}(R)/{\sim}$
\end{center}
where two matrices $M_{1} \in A_{2n}(R), M_{2} \in A_{2m}(R)$ are equivalent if
\begin{center}
$M_1 \perp \psi_{2m+2s} = \varphi^{t}(M_2 \perp \psi_{2n+2s}) \varphi$
\end{center}
for some $s \in \mathbb{N}$ and $\varphi \in SL_{2n+2m+2s}(R)$. It follows from \cite[\S 3]{SV} that the direct sum of matrices equips this set with the structure of an abelian group. Any ring homomorphism $R \rightarrow S$ induces a group homomorphism $W_{SL}(R) \rightarrow W_{SL}(S)$.\\
Now recall from \cite{Sy1} that there exists a map called  generalized Vaserstein symbol modulo SL
\begin{center}
$V_{\theta_{0}}: Um_{3}(R)/SL_{3}(R) \rightarrow W_{SL}(R)$
\end{center}
associated to the isomorphism $\theta_{0}: R \xrightarrow{\cong} \det (R^2), 1 \mapsto e_{1}\wedge e_{2}$, where $e_{1}=(1,0), e_{2}=(0,1) \in R^2$. The orbit space $Um_{3}(R)/SL_{3}(R)$ was studied by means of this map in \cite{Sy1}. In particular, the following criterion for the triviality of $Um_{3}(R)/SL_{3}(R)$ was proven as a special case of \cite[Corollary 3.7]{Sy1}: For a Noetherian ring $R$ of dimension $\leq 4$ such that $SL_{i} (R)$ acts transitively on $Um_{i} (R)$ for $i=4,5$, the orbit space $Um_{3}(R)/SL_{3}(R)$ is trivial if and only if $W_{SL}(R)$ is trivial and $Sp_{4}(R)$ acts transitively on $Um_{4}(R)$.\\
It is well-known that stably free modules of rank $2$ over smooth affine algebras of dimension $3$ over algebraically closed fields with characteristic $\neq 2$ are free (cf. \cite{F3}, \cite{FRS}). For general (not necessarily smooth) affine algebras over algebraically closed fields of characteristic $\neq 2$, no such theorem was proven so far. The following theorem gives a precise criterion for such algebras:

\begin{Thm}\label{T5}
Let $R$ be an affine algebra of dimension $3$ over an algebraically closed field $k$ with characteristic $\neq 2$. Then all stably free $R$-modules of rank $2$ are free if and only if $W_{SL} (R) = 0$.
\end{Thm}

\begin{proof}
Let us first assume that $R$ is reduced. Then the statement follows from Theorem \ref{T3} and the preceding paragraphs as $SL_{i}(R)$ clearly acts transitively on $Um_{i}(R)$ for $i=4,5$ as a consequence of \cite[Theorem 1]{S1} and \cite[Chapter IV, Theorem 3.4]{HB}.\\
Now assume that $R$ is a general (not necessarily reduced) affine algebra of dimension $3$ over an algebraically closed field $k$ with characteristic $\neq 2$. Let $J$ be the nilradical of $R$ and set $\overline{R} = R/J$. Then $SL_{5}(R)$ still acts transitively on $Um_5 (R)$ by  \cite[Chapter IV, Theorem 3.4]{HB}; so the generalized Vaserstein symbol
\begin{center}
$V_{\theta_{0}}: Um_{3}(R)/SL_{3}(R) \rightarrow W_{SL}(R)$
\end{center}
is surjective by \cite[Theorem 3.2]{Sy1}. In particular, if all stably free $R$-modules of rank $2$ are free, the orbit space $Um_{3}(R)/SL_{3}(R)$ is trivial and hence the group $W_{SL}(R)$ is trivial.\\
Conversely, assume $W_{SL}(R)$ is trivial. Then we claim that also $W_{SL}(\overline{R})$ is trivial; indeed, the homomorphism $W_{SL}(R) \rightarrow W_{SL}(\overline{R})$ induced by the projection $R \rightarrow \overline{R}$ is surjective: If $\overline{M}$ is an invertible alternating matrix of rank $2n$ over $\overline{R}$ with Pfaffian $1$, it is easy to see that there is an alternating $2n \times 2n$-matrix $M'$ over $R$ which agrees with $\overline{M}$ modulo $J$. In particular, the Pfaffian of the matrix $M'$ will be of the form $1+x$, where $x \in J$ is a nilpotent element; hence $1+x$ is a unit and $M'$ is invertible. Then the matrix
\begin{center}
$M={({(1+x)}^{-1}\cdot id_{1} \perp id_{2n-1})}^{t} M' ({(1+x)}^{-1}\cdot id_{1} \perp id_{2n-1})$
\end{center}
will be invertible alternating with Pfaffian $1$ and will still agree with $\overline{M}$ modulo $J$. Altogether, any invertible alternating matrix with Pfaffian $1$ over $\overline{R}$ has a lift to an invertible alternating matrix with Pfaffian $1$ over $R$ and hence the homomorphism $W_{SL}(R) \rightarrow W_{SL}(\overline{R})$ is indeed surjective and $W_{SL}(\overline{R})$ is indeed trivial as soon as $W_{SL}(R)$ is trivial, as claimed.\\
In particular, all stably free $\overline{R}$-modules of rank $2$ are then free by the first paragraph of this proof. If $P$ is then a stably free $R$-module of rank $2$, we have an isomorphism $P \otimes_R {\overline{R}} \cong \overline{R}^2$. But then it is a well-known fact (e.g., cf. \cite[Chapter I, Corollary 1.6]{L}) that there is also an isomorphism $P \cong R^2$. So all stably free $R$-modules of rank $2$ are free. This completes the proof. 
\end{proof}

The following result extends \cite[Theorem 3.19]{Sy1} to normal affine algebras:

\begin{Thm}\label{T6}
Let $R$ be a normal affine algebra of dimension $4$ over an algebraically closed field $k$ with $6 \in k^{\times}$. Then all stably free $R$-modules of rank $2$ are free if and only if $W_{SL} (R) = 0$.
\end{Thm}

\begin{proof}
This follows from Theorem \ref{T4} and the paragraph preceding Theorem \ref{T5} as a consequence of \cite[Theorem 7.5]{FRS} and \cite[Theorem 1]{S2}.
\end{proof}


\begin{thebibliography}{xxxxxx}
\bibitem[AF]{AF} A. Asok, J. Fasel, A cohomological classification of vector bundles on smooth affine threefolds, Duke Math. Journal 163 (2014), no. 14, 2561-2601
\bibitem[AFH]{AFH} A. Asok, J. Fasel, M. J. Hopkins, Algebraic vector bundles and $p$-local $\mathbb{A}^1$-homotopy theory, preprint available at https://arxiv.org/abs/2008.03363
\bibitem[AHW]{AHW} A. Asok, M. Hoyois, M. Wendt, Affine representability results in $\mathbb{A}^{1}$-homotopy theory II: principal bundles and homogeneous spaces, Geom. Top. 22 (2018), no. 2, 1181-1225
\bibitem[F1]{F1} J. Fasel, Projective modules over the real algebraic sphere of dimension $3$, J. Algebra 325 (2011), no.1, 18-33
\bibitem[F2]{F2} J. Fasel, Some remarks on orbit sets of unimodular rows, Comment. Math. Helv. 86 (2011), no. 1, 13-39
\bibitem[F3]{F3} J. Fasel, Stably free modules over smooth affine threefolds, Duke Math. Journal 156 (2011), no. 1, 33-49
\bibitem[FRS]{FRS} J. Fasel, R. A. Rao and R. G. Swan, On stably free modules over affine algebras, Publ. Math. Inst. Hautes \'Etudes Sci. 116 (2012), 223-243
\bibitem[G]{G} A. Gupta, Optimal injective stability for the symplectic $K_{1}{Sp}$ group, J. Pure Appl. Algebra 219 (2015), 1336-1348
\bibitem[HB]{HB} H. Bass, Algebraic K-theory, Benjamin, New York, 1968
\bibitem[Ho]{Ho} M. Hovey, Model categories, volume 63 of Mathematical Surveys and Monographs, American Mathematical Society, Providence, 1999
\bibitem[L]{L} T. Y. Lam, Serre's problem on projective modules, Springer Monographs in Mathematics, Springer, Berlin, Heidelberg, New York, 2006
\bibitem[Mo]{Mo} F. Morel, $\mathbb{A}^{1}$-algebraic topology over a field, volume 2052 of Lecture Notes in Mathematics, Springer, Heidelberg, 2012
\bibitem[MV]{MV} F. Morel, V. Voevodsky, $\mathbb{A}^{1}$-homotopy theory of schemes, Inst. Hautes \'Etudes Sci. Publ. Math. 90 (1999), 45-143 
\bibitem[NMK]{NMK} N. M. Kumar, Stably free modules, American Journal of Mathematics 107 (1985), no.6, 1439-1444
\bibitem[S1]{S1} A. A. Suslin, On stably free modules, Math. U.S.S.R. Sbornik (1977), 479-491
\bibitem[S2]{S2} A. A. Suslin, A cancellation theorem for projective modules over algebras, Dokl. Akad. Nauk SSSR 236 (1977), 808-811
\bibitem[Sch]{Sch} W. Scharlau, Quadratic and Hermitian forms, Grundlehren der Mathematischen Wissenschaften 270, Springer-Verlag, Berlin, 1985
\bibitem[SV]{SV} A. A. Suslin, L. N. Vaserstein, Serre's problem on projective modules over polynomial rings, and algebraic K-theory, Izv. Akad. Nauk. SSSR Ser. Mat. 40 (1976), 993-1054
\bibitem[SwT]{SwT} R. G. Swan, J. Towber, A class of projective modules which are nearly free, J. Algebra 36 (1975), 427-434
\bibitem[Sy1]{Sy1} T. Syed, The cancellation of projective modules of rank 2 with a trivial determinant, Algebra \& Number Theory 15 (2021), no. 1, 109-140
\bibitem[Sy2]{Sy2} T. Syed, Symplectic orbits of unimodular rows, Journal of Algebra (2024), Volume 646, 478-493
\bibitem[V]{V} L. N. Vaserstein, Operations on orbits of unimodular vectors, J. Algebra 100 (1986), 456-461
\end{thebibliography}
\end{document}